\documentclass[dvipdfmx]{amsart}
\usepackage{amsmath,amssymb,bm}
\usepackage[dvipdfmx]{graphicx}
\usepackage{color}
\usepackage{comment}
\theoremstyle{definition}
\newtheorem{defn}{Definition}[section]
\newtheorem{example}[defn]{Example}
\newtheorem{rem}[defn]{Remark}
\theoremstyle{plain}
\newtheorem{thm}[defn]{Theorem}

\newtheorem{lem}[defn]{Lemma}
\newtheorem{cor}[defn]{Corollary}

\newtheorem{question}[defn]{Question}
 \makeatletter
    
    \@addtoreset{equation}{section}
    \@namedef{subjclassname@2020}{%
  \textup{2020} Mathematics Subject Classification}
  \makeatother

\title{Flat plumbing basket and contact structure}
\author{Tetsuya Ito}
\address{Department of Mathematics, Kyoto University, Kyoto 606-8502, JAPAN}
\email{tetitoh@math.kyoto-u.ac.jp}
\author{Keiji Tagami}
\address{
Department of Fisheries Distribution and Management,
National Fisheries University,
Shimonoseki, Yamaguchi 759-6595
JAPAN
}
\email{tagami@fish-u.ac.jp}
\date{\today}
\begin{document}
\maketitle
\begin{abstract}
A flat plumbing basket is a Seifert surface consisting of a disk and bands contained in distinct pages of the disk open book decomposition of the 3-sphere. 
In this paper, we examine close connections between flat plumbing baskets and the contact structure supported by the open book.
As an application we give lower bounds for the flat plumbing basket numbers and determine the flat plumbing basket numbers for various knots and links, including the torus links. 
\end{abstract}
\section{Introduction}
In \cite{FHK} Furihata-Hirasawa-Kobayashi showed that every oriented link in $\mathbf{S}^{3}$ bounds a certain special Seifert surface which they called a \emph{flat plumbing basket}. A surface $F$ in $\mathbf{S}^{3}$ is a flat plumbing basket if it is obtained from a disk by plumbing unknotted and untwisted annuli so that the gluing regions are contained in the disk. We call a flat plumbing basket $F$ whose boundary is $L$ \emph{a flat plumbing basket presentation} of $L$.

\par
Equivalently, a flat plumbing basket can be seen as a surface which consists of a page of the disk open book decomposition of $\mathbf{S}^{3}$ and finitely many bands contained in distinct pages of the open book. In this point of view, it is quite natural to expect a connection between flat plumbing baskets and contact geometry since an open book decomposition is a decomposition of \emph{contact} 3-manifolds.
\par 
In this paper, we discuss a connection between flat plumbing basket presentations and contact geometry. We show that a flat plumbing basket presentation can be seen as Legendrian and transverse link representatives in several natural ways. 
\par 
As an application of a contact geometry point of view, we give various lower bounds for the {\it flat plumbing basket number} $fpbk(L)$, the minimal number of bands which is needed to obtain a flat plumbing basket presentation of the link $L$ \cite{Kim-flat}. 
Let $\overline{L}$ be the mirror image of $L$. Then we obtain the following lower bounds for $fpbk(L)$.
\begin{thm}\label{thm:self-linking}
Let $L$ be an oriented link in $\mathbf{S}^3$. 
Then 
\[
fpbk(L)\geq \max\{-\overline{sl}(L), -\overline{sl}(\overline{L})\}-1.
\]
\end{thm}
\begin{thm}\label{thm:arc-00}
Let $L$ be a non-trivial oriented link in $\mathbf{S}^3$. 
Then
\[ 
2fpbk(L) \geq \max\{-\overline{tb}(L),-\overline{tb}(\overline{L})\}+b(L).
\]
\end{thm}
Here $\overline{sl}(L)$ and $\overline{tb}(L)$ are the maximal self-linking number and the maximal Thurston-Bennequin number of $L$, and $b(L)$ denotes the braid index of $L$. It turns out these lower bounds are quite useful. We determine the flat plumbing basket number for various knots, including the $(p,q)$-torus link $T_{p,q}$ (Example \ref{exam:torus}). 
\par
This paper is organized as follows: 
In Section~\ref{sec:preliminary}, we recall the definition of a flat plumbing basket presentation of a link in $\mathbf{S}^{3}$.
\par 
In Section~\ref{sec:self-linking}, we show that a flat plumbing basket presentation can be viewed as a front projection of a Legendrian link $\mathcal{L}_{F}$. 
By considering the transverse push-off of $\mathcal{L}_{F}$ we prove Theorem~\ref{thm:self-linking}. 
We also point out that a flat plumbing basket presentation can be seen as a closed braid, which gives a simple and direct alternative proof of Theorem~\ref{thm:self-linking}. We then study when the lower bound in Theorem~\ref{thm:self-linking} is sharp, and determine the flat plumbing basket number for various knots.
\par 
In Section~\ref{sec:arc}, we discuss another way to relate a flat plumbing basket presentation and a Legendrian link. 
We explain how to view a flat plumbing basket presentation as an arc presentation and how to get the corresponding grid diagram. This gives a different Lengendrian link $\mathcal{L}_G$, which leads to Theorem~\ref{thm:arc-00}.
\par 
\par
Section \ref{sec:tabulation} contains tables of the flat plumbing basket numbers for prime knots with less than or equal to 9 crossings, improving a previous table \cite[Table~1]{Hirose-Nakashima}. 
We determine the flat plumbing basket numbers for these knots, with 6 exceptions.
%
\par 
Throughout this paper, all links are oriented. For basics of Legendrian knot theory, we refer to \cite{book-contact}.
This paper is an expanded and improved version of second author's preprint \cite{tagami4}.
\section{Flat plumbing basket in $\mathbf{S}^{3}$}\label{sec:preliminary}
Let $U$ be the oriented unknot in $\mathbf{S}^3$ and $\pi\colon \mathbf{S}^3\setminus U\rightarrow \mathbf{S}^1 = [0,2\pi] \slash (0\!\!\sim\!\!2\pi)$ be the disk fibration of its complement. 
We call the pair $(U,\pi)$ the {\it disk open book decomposition} of $\mathbf{S}^3$. 
The closure of a fiber $D_{\theta}:=\overline{\pi^{-1}(D_{\theta})}$ is called  a \emph{page}, and $U$ is called the \emph{binding}. 
%
%
%
\par
A Seifert surface $F$ is a {\it flat plumbing basket} if there are finitely many disjoint bands $B_{1}, \dots, B_{n}$ and $0<\theta_{1}<\dots<\theta_{n}<2\pi$ such that 
$F=D_{0}\cup B_{1} \cup \dots \cup B_{n}$, each band $B_{i}$ is contained in $D_{\theta_{i}}$ and $B_{i}\cap U$ consists of two arcs, where $n=b_{1}(F)$. 
\par
We call a flat plumbing basket $F$ whose boundary is an oriented link $L$ a {\it flat plumbing basket presentation} of $L$ (see Figure~\ref{fig:fpb1}). 
\par 
The {\it flat plumbing basket number} $fpbk(L)$ of an oriented link $L$ is the minimal number of bands among all flat plumbing basket presentations of $L$.
Namely, 
\[
fpbk(L):=\min\{b_{1}(F)\mid F\text{ is a flat plumbing basket presentation of }L\}, 
\]
where $b_{1}(F)$ is the first betti number of $F$. 
We remark that $fpbk(L)-|L|+1\in 2\mathbf{Z}_{\geq 0}$, where $|L|$ is the number of the components of $L$, and that $fpbk$ is preserved under taking the mirror image. 
\begin{figure}[h]
\begin{center}
\includegraphics[scale=0.55]{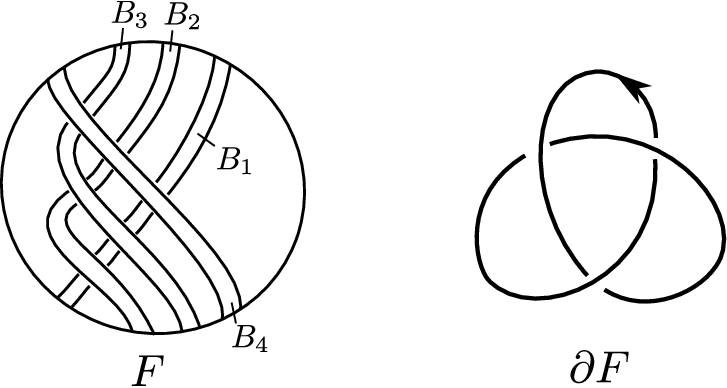}
\end{center}
\caption{An example of a flat plumbing basket $F$. It is a flat plumbing basket presentation of the negative trefoil. 
}
\label{fig:fpb1}
\end{figure}
%
\section{Flat plumbing basket as front projection}\label{sec:self-linking}
Let $F$ be a flat plumbing basket with bands $B_1, \ldots, B_n$. 
We draw $F$ as a union of bands and a slanted rectangle 
 so that the bands are attached to the rectangle at the top left edge (see the middle picture in Figure~\ref{fig:front}). 
Moreover, we draw each band so that 
\begin{itemize}
\item it is sufficiently thin and 
\item the core of the band consists of three line segments: one is parallel to the top left edge of the rectangle and the other two are perpendicular to the top left edge. 
\end{itemize}

We view as $F \subset \mathbf{R}^{3}$ and let $p:\mathbf{R}^{3} \rightarrow \mathbf{R}^{2}$ be the projection. 
Let $h_i$ be the distance of $p(\mbox{the top left edge of the rectangle})$ and $p($the line segment of the core of $B_i$
parallel to the edge$)$ in $\mathbf{R}^{2}$.
We deform each band so that $h_i$ satisfies $h_i>h_j$ if $i>j$. 
Then, 
the link diagram $p(L)=p(\partial F)$ has the property that at each crossing, the strand with smaller slope always lies in front. 
After replacing the local maxima and local minima in the horizontal coordinate with cusps, and smoothening the local maxima and local minima in the vertical coordinate, we obtain a front projection of a Legendrian link in the standard contact structure $\xi_{std}$ on $\mathbf{S}^3$ (see the right picture in Figure~\ref{fig:front} -- This picture answers \cite[Question~6.2]{tagami4}). 
\par 
We denote the Legendrian link by $\mathcal{L}_{F}$ and call it the {\it Legendrian link associated with $F$}. 
\begin{figure}[h]
\begin{center}
\includegraphics[scale=0.6]{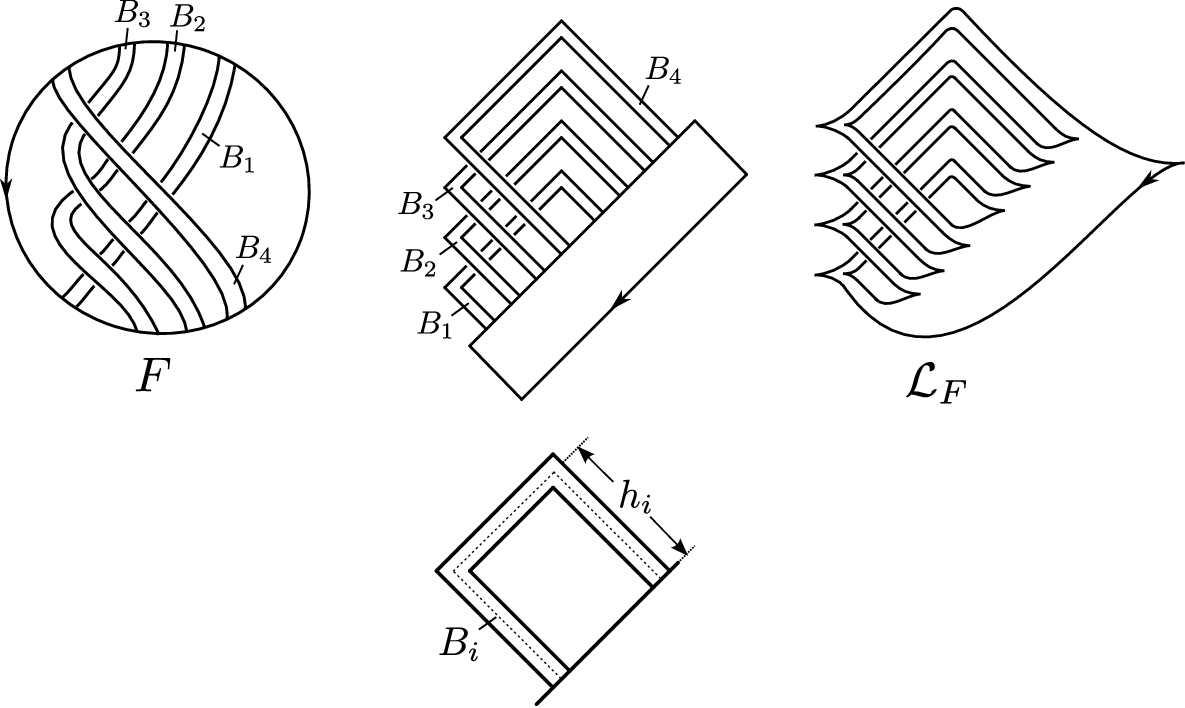}
\end{center}
\caption{
How to draw a front projection from a flat plumbing basket:
For a flat plumbing basket $F$, we draw $F$ as the middle picture, so that  $h_{i}>h_{j}$ holds if $i>j$. 
With suitable modification such a diagram can be naturally regarded as a front projection of a Legendrian link $\mathcal{L}_{F}$ (the right picture). 
}
\label{fig:front}
\end{figure}

\par 
One can immediately read the Thurston-Bennequin number and the rotation number of $\mathcal{L}_{F}$ from the front projection.
\begin{lem}\label{lem:thurston-bennequin}
Let $F$ be a flat plumbing basket with $b_1(F)>0$ and $\mathcal{L}_{F}$ be the Legendrian link associated with $F$. Then we obtain 
\begin{align}
tb(\mathcal{L}_{F})&=-2b_{1}(F), \\
rot(\mathcal{L}_{F})&=-b_{1}(F)+1. 
\end{align}
\end{lem}
\begin{proof}
By construction of the front projection, the front projection of $\mathcal{L}_{F}$ has $2b_1(F)$ right cusps and the sum of signs of crossings is zero. 
Hence we have 
$tb(\mathcal{L}_{F})=0-2b_1(F)=-2b_1(F). $
\par 
The front projection of $\mathcal{L}_{F}$ has $2b_1(F)-1$ up right cusps and $1$ down right cusp, and the numbers of up and down left cusps are equal so $rot(\mathcal{L}_{F})=\frac{1}{2}(1-(2b_1(F)-1))=-b_1(F)+1$.
\end{proof}
\par 
\begin{proof}[Proof of Theorem~\ref{thm:self-linking}]
If $b_{1}(F)=0$, we see $\partial F=U$ and the equality of Theorem~\ref{thm:self-linking} holds. Hence we can suppose that $b_{1}(F)>0$.
Let $\mathcal{T}_F$ be the transverse positive push-off of $\mathcal{L}_F$. By Lemma~\ref{lem:thurston-bennequin} we have
\begin{equation}
\label{eqn:slF}
sl(\mathcal{T}_F)=tb(\mathcal{L}_F)-rot(\mathcal{L}_F)=-b_1(F)-1. 
\end{equation}
Let $\overline{sl}(L)$ be the maximal self-linking number of the topological link $L$, that is,
\[ 
\overline{sl}(L)=\max \{sl(\mathcal{T}) \: | \: \mathcal{T} \mbox{ is a transverse link topologically isotopic to } L\}. 
\]
For a flat plumbing basket presentation $F$ of $L$, we have obvious inequalities $ \overline{sl}(L)\geq sl(\mathcal{T}_F)$ and $fpbk(L) \leq b_{1}(F)$.
Also, $fpbk(L)=fpbk(\overline{L})$, where  $\overline{L}$ denotes the mirror image of $L$. Hence, by $(\ref{eqn:slF})$, we finish the proof. 
\end{proof}
%

\par 
Here is an alternative, direct proof of Theorem~\ref{thm:self-linking} which is interesting in its own right.
For a flat plumbing basket $F$, we view $L=\partial F$ as a closed $(b_{1}(F)+1)$-braid by viewing each band as a union of a disk and two twisted bands with opposite twisting (see Figure~\ref{fig:fpbktobraid}). We note that the exponent sum of such a braid is $0$.
This closed braid provides a transverse link $\mathcal{T}_{br}$. 
By Bennequin's formula \cite{Bennequin} $sl(\mathcal{T}_{br})=-b_{1}(F)-1$. 

%
%
\begin{figure}[htbp]
\begin{center}
\includegraphics*[width=65mm]{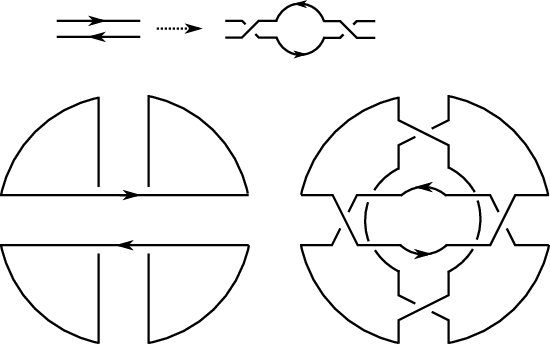}
   \caption{Flat plumbing basket, viewed as closed ($b_{1}(F)+1$)-braid}
 \label{fig:fpbktobraid}
\end{center}
\end{figure}
\begin{rem}
The transverse link $\mathcal{T}_{br}$ is transverse isotopic to $\mathcal{T}_{F}$, the transverse push-off of $\mathcal{L}_F$. This can be checked, for example, by converting the arc presentation from the front projection $\mathcal{L}_F$ into a closed braid representative, as we discuss in the next section. 
\end{rem}
\begin{cor}\label{cor:self-linking-homfly}
Let $L$ be an oriented link. Then we have 
\[
\operatorname{maxdeg}_{v}P_{L}(v,z)\leq fpbk(L), 
\]
where $P_L(v,z)$ is the HOMFLYPT polynomial and $\operatorname{maxdeg}_{v}$ is the maximal degree of the variable $v$. 
\end{cor}
\begin{proof}
The HOMFLYPT bound on the self-linking number (the Morton-Franks-Williams inequality \cite{F-W,Morton}) implies 
\[
\operatorname{maxdeg}_{v}P_{L}(v,z)\leq -\overline{sl}(L)-1. 
\]
Hence by Theorem~\ref{thm:self-linking}, we finish the proof. 
\end{proof}
\subsection{Applications and discussions}
It turns out Theorem~\ref{thm:self-linking} is quite useful, especially in the case the link $L$ is represented as a closure of a positive braid. 
\par
Let $b(L)$ be the braid index of $L$, the minimum number of strands needed to represent $L$ as a closed braid. Let $\beta$ be a $b(L)$-braid whose closure is $L$. 
By \cite{Dynnikov-Prasolov,LaFountain-Menasco},  a closed braid representative of the minimum braid index always attains the maximal self-linking number so we have
\begin{equation}\label{eqn:minimal}
\overline{sl}(L)= -b(L)+e(\beta) \text{ and } \ \overline{sl}(\overline{L})=-b(L)-e(\beta), 
\end{equation}
where $e(\beta)$ denotes the exponent sum of the braid $\beta$.
Consequently, we have
\begin{equation}\label{eqn:br}
 -\overline{sl}(L)-\overline{sl}(\overline{L})=2b(L). 
\end{equation}
\par 
On the other hand, by Bennequin's inequality \cite{Bennequin} we have 
\begin{equation}
\label{eqn:Be}
\overline{sl}(L) \leq -\chi(L), 
\end{equation}
where $\chi(L)$ is the maximum Euler characteristic of a Seifert surface of $L$.
Therefore we have the following consequence of Theorem~\ref{thm:self-linking} and equation $(\ref{eqn:br})$. 
\begin{cor}\label{cor:B}
Let $L$ be an oriented link such that the Bennequin inequality (\ref{eqn:Be}) is sharp, namely, $\overline{sl}(L) = -\chi(L)$ holds. Then $fpbk(L)\geq -\chi(L)+2b(L)-1$. 
\end{cor}

\par
It is interesting to compare this bound with other lower bounds of $fpbk(L)$.
First of all, a flat plumbing basket $F$ with $b_{1}(F)>0$ is always compressible with $\chi(F)=-b_{1}(F)+1$ so we have a lower bound
\begin{equation}
\label{eqn:genusbound}
fpbk(L) \geq -\chi(L)+3, 
\end{equation}
which we call the \emph{trivial genus bound}. Hirose-Nakashima \cite[Theorem~1.3]{Hirose-Nakashima} proved
 \begin{equation}
\label{eqn:HN}
\mbox{if } \Delta_{L}(t) \mbox{ is not monic, } fpbk(L) \geq \deg \Delta_{L}(t)+4, 
\end{equation}
where $\Delta_{L}(t)$ denotes the Alexander polynomial of $L$.
We call (\ref{eqn:HN}) \emph{the Hirose-Nakashima bound}.
\par 
Since $\deg \Delta_{L}(t)+4 \leq -\chi(L)+3$ and $b(L)\geq 2$ for any non-trivial link $L$, Corollary~\ref{cor:B} gives a better estimate than the trivial genus bound (\ref{eqn:genusbound}) and the Hirose-Nakashima bound (\ref{eqn:HN}) when the Bennequin inequality (\ref{eqn:Be}) is sharp.
\par 
%

Let $\sigma_{1},\ldots,\sigma_{n-1}$ be the standard (Artin) generator of the braid group $B_n$. An $n$-braid $\beta$ is \emph{positive} if $\beta$ is a product of $\sigma_{1},\ldots,\sigma_{n-1}$. 
Similarly an $n$-braid $\beta$ is \emph{strongly quasipositive} if $\beta$ is a product of so-called the \emph{band generators}
\[
a_{i,j} = (\sigma_{j-1} \cdots \sigma_{i+1} \sigma_{i})^{-1}\sigma_{j}(\sigma_{j-1} \cdots \sigma_{i+1} \sigma_{i})\quad (1\leq i<j \leq n). 
\]

A link $L$ is a \emph{strongly quasipositive} if it is represented as the closure of a strongly quasipositive braid. 
A band generator can be viewed as a positively twisted band connecting the $i$-th and the $j$-th strands of the braid. Consequently, when $L$ is strongly quasipositive we have a Seifert surface $F$ of $L$ with $\overline{sl}(L)=-\chi(F)$ so Bennequin's inequality (\ref{eqn:Be}) is sharp.
\par 
Hence Theorem~\ref{thm:self-linking} will be effective for strongly quasipositive links. In Tables~\ref{table-1}--\ref{table-last}, for prime knots $K$ with less than or equal to $9$ crossings we will see that $-\overline{sl}(\overline{K})-1=fpbk(K)$ if $K$ is strongly quasipositive.
This leads to the following question:
\par 
\begin{question}
\label{ques:sqp}
If a link $L$ is strongly quasipositive, 
is the inequality in Theorem~\ref{thm:self-linking} (equivalently, the inequality in Corollary~\ref{cor:B}) an equality? 
Namely, for such $L$, is it true that 
\[
fpbk(L)=-\overline{sl}(\overline{L})-1=-\chi(L)+2b(L)-1?
\]
\end{question}
\par 
As a partial answer, we show the following.
\par 
\begin{thm}\label{theorem:positive}
If $L$ is presented by the closure of a positive $b(L)$-braid, then
\[
fpbk(L)=-\overline{sl}(\overline{L})-1=-\chi(L)-1+2b(L). 
\] 
\end{thm}
\begin{proof}
By Corollary~\ref{cor:B} we have $fpbk(L) \geq-\overline{sl}(\overline{L})-1=-\chi(L)-1+2b(L)$ so we check the converse inequality.
\par 
Let $\beta$ be a positive $b(L)$-braid whose closure is $L$ and $\overline{\beta}$ be the mirror image of the braid $\beta$. 
Then $\overline{\beta}$ is a negative $b(L)$-braid whose closure is $\overline{L}$. Therefore, by $(\ref{eqn:minimal})$, we have
$\overline{sl}(\overline{L})=-b(L)+e(\overline{\beta})=-b(L)-e(\beta)$. 
Since $\overline{\beta}$ is a negative braid, it is a product of $e(\beta)$ standard negative generators $\{\sigma_{1}^{-1},\ldots,\sigma_{b(L)-1}^{-1}\}$.
Thus the mirror image $\overline{L}$ is a closure of a braid of the form
\begin{align*}
&(\sigma_{b(L)-1}\cdots \sigma_{1})(\sigma_{1}^{-1}\cdots \sigma_{b(L)-1}^{-1})\overline{\beta} \\
& = (\sigma_{b(L)-1}\cdots \sigma_{1})(\mbox{braid word consisting of } e(\beta)+(b(L)-1) \mbox{ negative generators}). 
\end{align*}
By \cite[Theorem~2.4]{FHK}, from such a closed braid representative we construct a flat plumbing basket presentation $F$ of $L$ with $b_1(F)=e(\beta)+b(L)-1=-\overline{sl}(\overline{L})-1$.
\end{proof}
\par
Although this does not fully answer Question \ref{ques:sqp}, it can be used to determine the flat plumbing basket for the torus links. 
\begin{example}[Torus links]
\label{exam:torus}
For any $p\geq q>1$, let $T_{(p,q)}$ be the $(p,q)$-torus link.
Since $b(T_{p,q})=q$ and $T_{p,q}$ is the closure of $q$-braid $(\sigma_{1}\cdots \sigma_{q-1})^{p}$, by Theorem~\ref{theorem:positive}, 
\[
fpbk(T_{p,q})=-\overline{sl}(T_{p,-q})-1=-\chi(T_{p,q})-1+2q=pq-p+q-1. 
\]
\end{example}
\par 
\begin{example}[Twist knots]
Let $K_{m}$ be the $m$-twist knot (Figure~\ref{fig:twist-knot}). 
By \cite[Theorem~1.2]{ENV} and Theorem~\ref{thm:self-linking}, 
\[ 
2k\leq fpbk(K_{2k}), \ 2k+4\leq fpbk(K_{2k+1}). 
\]
Mikami Hirasawa shows that $fpbk(K_{2k+1})\leq 2k+4$ for $k\geq 0$, and $fpbk(K_{2k})\leq 2k$ for $k\geq 3$ in his forthcoming paper. Hence 
\begin{itemize}
\item $fpbk(K_{2k+1})=2k+4$ for $k\geq 0$, 
\item $fpbk(K_{2k})=2k$ for $k\geq 3$, 
\item $fpbk(K_{2})=4$ and $fpbk(K_{4})=6$. 
\end{itemize}
\begin{figure}[h]
\begin{center}
\includegraphics[scale=0.85]{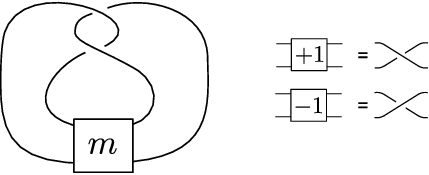}
\end{center}
\caption{The $m$-twist knot $K_{m}$ }
\label{fig:twist-knot}
\end{figure}
\end{example}
\par
The flat plumbing basket number is subadditive under the connected sum of knots but it is not additive in general. 
For example, Hirose-Nakashima \cite[Remark~1.4(b)]{Hirose-Nakashima} proved that $fpbk(3_1)=fpbk(\overline{3_1})=4$ but $fpbk(3_1\sharp \overline{3_1})=6$. 
Nao Kobayashi (Imoto) proved that $fpbk(6_1)=fpbk(\overline{6_1})=6$ but $fpbk(6_1\sharp \overline{6_1})=8$ in \cite[Proposition~5.4]{imoto-thesis}.
\par 
Theorem~\ref{thm:self-linking} gives some sufficient conditions for $fpbk$ to be additive.
\begin{example}\label{ex:additivity}
For knots $K$ and $K'$, $\overline{sl}(K\sharp K') = \overline{sl}(K)+\overline{sl}(K')+1$ holds. 
On the other hand, $fpbk$ is subadditive. 
This shows that if $K_1,\ldots,K_n$ satisfies $fpbk(K_i)=-\overline{sl}(K_{i})-1$ then, by Theorem~\ref{thm:self-linking}, the following holds: 
\[
fpbk(K_1 \sharp K_2 \sharp \cdots \sharp K_n)=fpbk(K_1)+fpbk(K_2)+\cdots + fpbk(K_n). 
\]
For example, $fpbk(3_1\sharp 3_1)=fpbk(3_1)+fpbk(3_1)=8$. 
\end{example}
\section{Flat plumbing basket and arc presentation}\label{sec:arc}
In this section, we give another natural way to relate flat plumbing baskets and Legendrain links.
For basics of arc presentation we refer to \cite{Cromwell2}. See \cite{Ng-Thurston} for a relation to Legendrian and transverse links and arc presentation. 
\par
As the definition already suggests, a flat plumbing basket presentation can be naturally seen as an arc presentation as follows. We view a flat plumbing basket $F$ as a union of the disk $D^{2}$ and $n=b_{1}(F)$ bands $B_1,\ldots,B_n$.
Then the link $L=\partial F$ is decomposed into two types of arcs: arcs that are subarcs of the boundary of bands $\partial B_i$, and arcs that are subarcs of $\partial D^{2}$.
We collapse $\partial D^{2} \cap L$, the latter types of subarcs, to points. Then the resulting link is isotopic to $L$ and consists of arcs coming from $\partial B_i$. We slightly move these arcs so that they lie on different pages.  Then we get an arc presentation $\mathcal{A}_F$ of $L$ with arc number $2n$ (See Figure \ref{fig:fpbktoarc}).

This immediately leads to the following.
\par 
\begin{figure}[htbp]
\begin{center}
\includegraphics*[width=120mm]{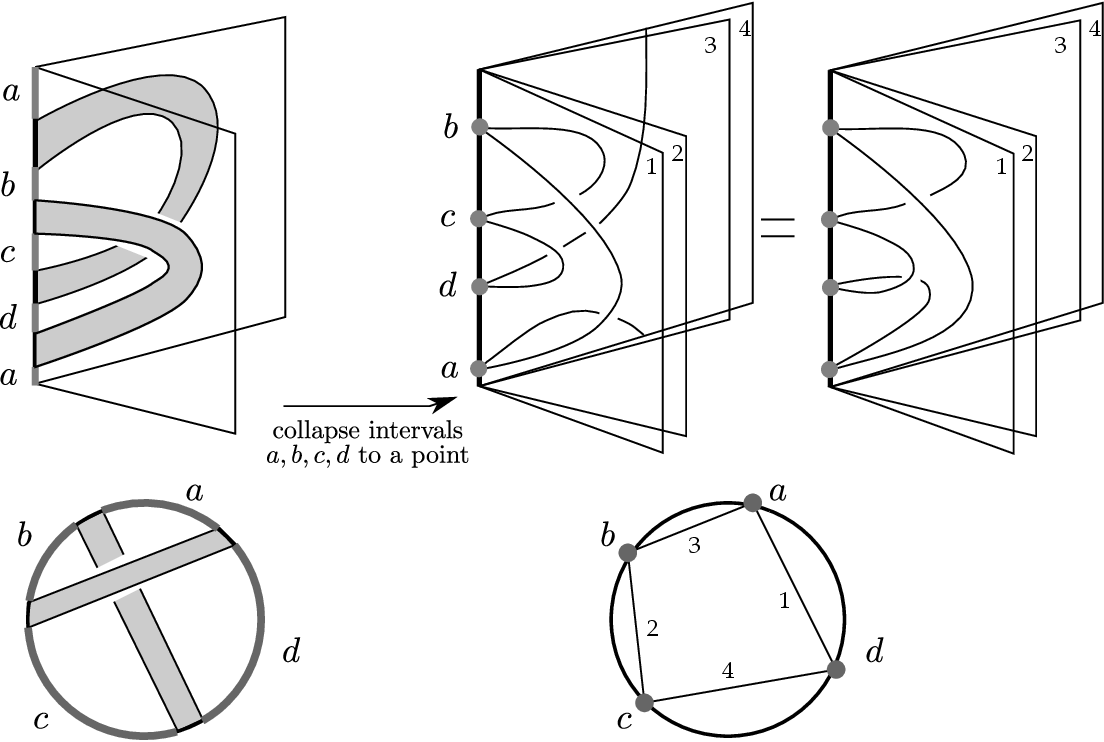}
   \caption{From flat plumbing basket to arc presentation:
 the labels $1,2,3,4$ represents the pages that contains the corresponding arcs}
 \label{fig:fpbktoarc}
\end{center}
\end{figure}
\par
\begin{thm}\label{thm:arc0}
Let $\alpha(L)$ be the arc index of $L$. Then 
\[
2fpbk(L) \geq \alpha(L). 
\]
\end{thm}
\par 
Since $\alpha(L)=c(L)+2$ for non-split alternating link $L$ \cite{Bae-Park}, where $c(L)$ denotes the crossing number of $L$, we have 
\begin{cor}
Let $L$ be a non-split alternating link. Then
\[
2fpbk(L)\geq c(L)+2. 
\]
\end{cor}
\par 
It seems likely that this inequality holds for general non-split links.
\par 

Unfortunately, this lower bound is not so good. To extract more detailed information, let us consider the grid diagram representative $G_{F}$ of the arc presentation $\mathcal{A}_{F}$, which is obtained as follows. 

From a construction of the arc presentation $\mathcal{A}_F$, a part of $L=\partial F$ that comes from a boundary of a band $B_i$ appears as oppositely oriented vertical line segments in $G_{F}$ (Figure~\ref{fig:fpbkgrid}(i)). 
By connecting the endpoints of these vertical lines by horizontal lines that correspond to the segment of binding connecting corresponding corners of bands, we get a grid diagram on the annulus (binding $U$)$\times [0,1]$. 
To get a usual grid diagram on rectangle, finally we flip one vertical line so that it is contained in the rectangle (Figure~\ref{fig:fpbkgrid}(ii), and see also the middle and the rightmost picture of Figure~\ref{fig:fpbktoarc}).

\begin{figure}[htbp]
\begin{center}
\includegraphics*[width=103mm]{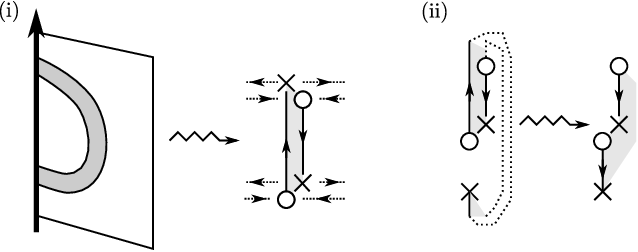}
\end{center}
\caption{From flat plumbing basket to grid diagram}
\label{fig:fpbkgrid}
\end{figure}

\begin{example}
\label{exam:fpbtogrid}

Figure \ref{fig:example-fpbkgrid} gives an example illustration of how to obtain a grid diagram from a flat plumbing basket. (i) depicts a flat plumbing basket, viewed in the open book. (ii) gives a resulting grid diagram on annulus. Each plumbed annulus gives rise to a pair of oppositely oriented vertical segments. Each segment of the binding connecting the corners of the bands appears as horizontal segment. For example, the horizontal segment $(*)$ in (ii) corresponds to a segment $(*)$ connecting the corner of the band $a$ and $d$ in (i). 
Finally, (iii) gives the grid diagram of rectangle, obtained by flipping. 

\begin{figure}[htbp]
\begin{center}
\includegraphics*[width=103mm]{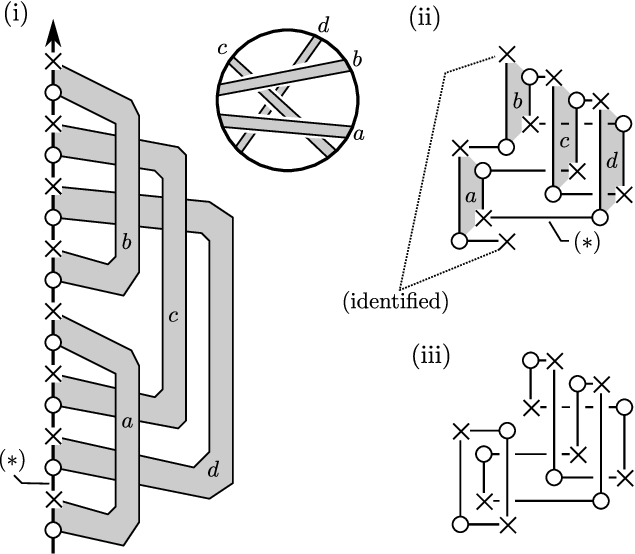}
\end{center}
\caption{From flat plumbing basket to grid diagram}
\label{fig:example-fpbkgrid}
\end{figure}
\end{example}

\par
Let $\mathcal{L}_{G}$  and $\mathcal{T}_{G}$
 be the Legendrian and the transverse 
links represented by the Grid diagram $G_{F}$. 
The Thurston-Bennequin number and the rotation number of $\mathcal{L}_{G}$ are given by
\[ 
tb(\mathcal{L}_{G})= \textrm{writhe} - \frac{1}{2}\# (\includegraphics*[width=3mm]{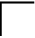},\includegraphics*[width=3mm]{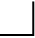}), \quad 
rot(\mathcal{L}_{G})=\frac{1}{2}\left( \# (\includegraphics*[width=3mm]{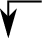},\includegraphics*[width=3mm]{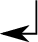}) -\# (\includegraphics*[width=3mm]{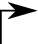},\includegraphics*[width=3mm]{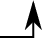}) \right)
\]
where $\# (\includegraphics*[width=3mm]{ne.eps},\includegraphics*[width=3mm]{sw.eps}),\# (\includegraphics*[width=3mm]{dcusp1.eps},\includegraphics*[width=3mm]{dcusp2.eps}) $ and $\# (\includegraphics*[width=3mm]{ucusp1.eps},\includegraphics*[width=3mm]{ucusp2.eps}) $, denote the numbers of corners of the indicated shapes in the grid diagram $G_{F}$.
\par
For each band $B_i$, there are $2^{4}=16$ possibilities of the orientations of horizontal lines from the endpoints of the pair of vertical segments. 
We classify the bands of $F$ into the following five types $A,B,C,D$ and $E$, according to its contribution to the Thurston-Bennequin number and the rotation number, as we show in Table~\ref{table:localpicture}. 
In Example \ref{exam:fpbtogrid}, the bands $a,b,d$ are of type $C$ and the band $c$ is of type $A$.
\par
%
\begin{table}[h]
\begin{tabular}{|c|c|c|c|c|c|c|c|}
\hline
Type&local picture&$tb$&$rot$ \\ \hline\hline
&&& \\
A&\mbox{\raisebox{-4mm}{\includegraphics[scale=0.15]{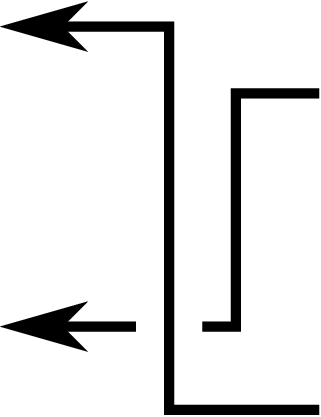}}}&$0$&$1$ \\
&&& \\ \hline
&&& \\
B&
\mbox{\raisebox{-4mm}{\includegraphics[scale=0.15]{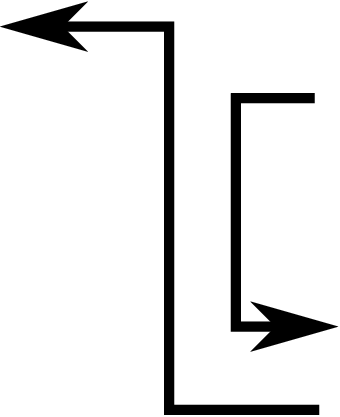}}}\ ,  \ 
\mbox{\raisebox{-4mm}{\includegraphics[scale=0.15]{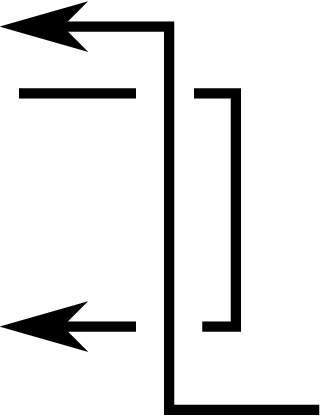}}}\ ,  \ 
\mbox{\raisebox{-4mm}{\includegraphics[scale=0.15]{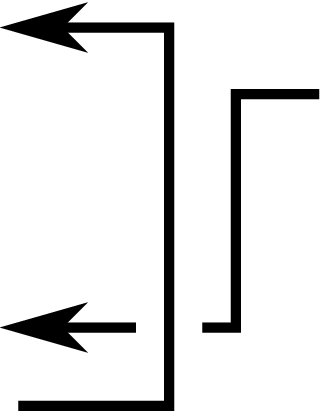}}}\ ,  \ 
\mbox{\raisebox{-4mm}{\includegraphics[scale=0.15]{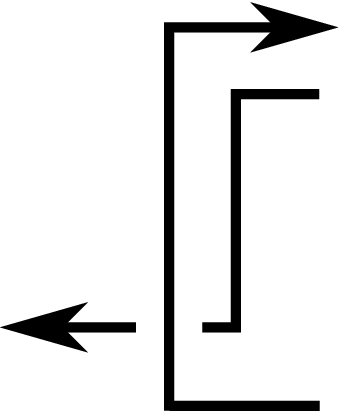}}}&$-\displaystyle{\frac{1}{2}}$&$\displaystyle{\frac{1}{2}}$ \\ 
&&& \\ \hline
&&& \\
C&
\mbox{\raisebox{-4mm}{\includegraphics[scale=0.15]{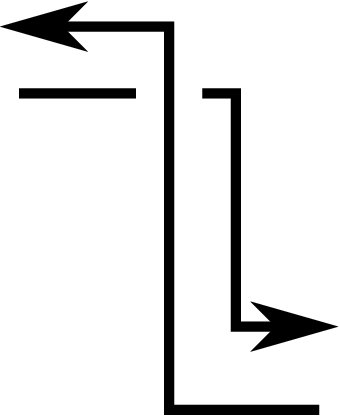}}}\ , \ 
\mbox{\raisebox{-4mm}{\includegraphics[scale=0.15]{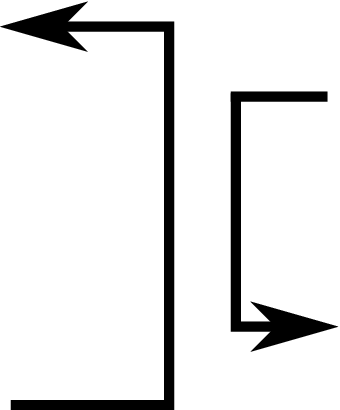}}}\ , \ 
\mbox{\raisebox{-4mm}{\includegraphics[scale=0.15]{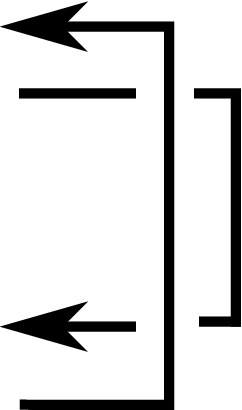}}}\ , \ 
\mbox{\raisebox{-4mm}{\includegraphics[scale=0.15]{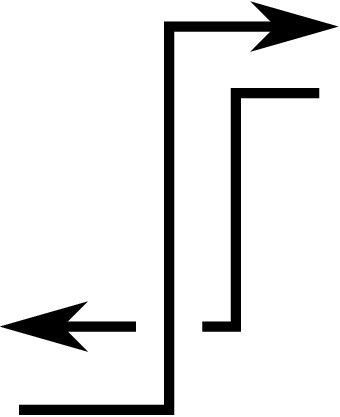}}}\ , \ 
\mbox{\raisebox{-4mm}{\includegraphics[scale=0.15]{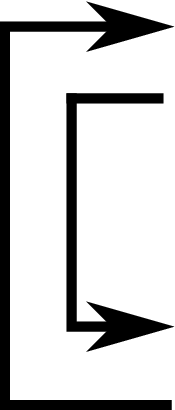}}}\ , \ 
\mbox{\raisebox{-4mm}{\includegraphics[scale=0.15]{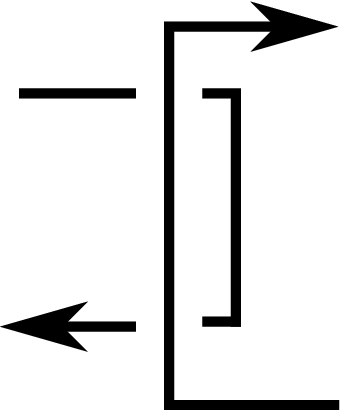}}}
&$-1$&$0$ \\ 
&&& \\ \hline
&&& \\
D&
\mbox{\raisebox{-4mm}{\includegraphics[scale=0.15]{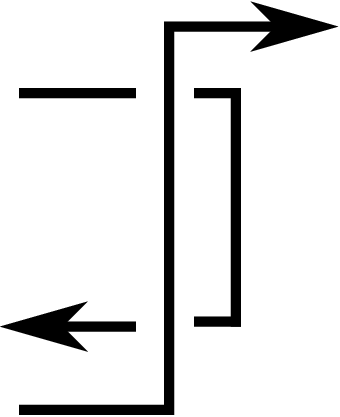}}}\ ,  \ 
\mbox{\raisebox{-4mm}{\includegraphics[scale=0.15]{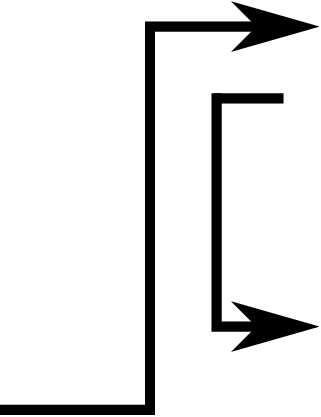}}}\ ,  \ 
\mbox{\raisebox{-4mm}{\includegraphics[scale=0.15]{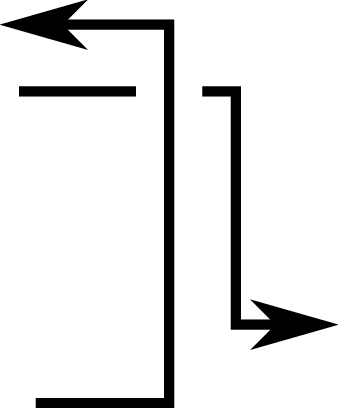}}}\ ,  \ 
\mbox{\raisebox{-4mm}{\includegraphics[scale=0.15]{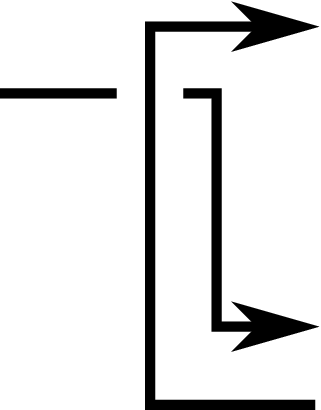}}} &$-\displaystyle{\frac{3}{2}}$&$-\displaystyle{\frac{1}{2}}$ \\
&&& \\ \hline
&&& \\
E&\mbox{\raisebox{-4mm}{\includegraphics[scale=0.15]{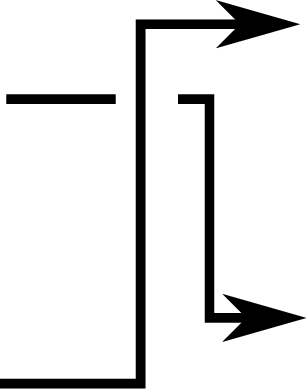}}}&$-2$&$-1$ \\ 
&&& \\ \hline
\end{tabular}
\caption{Fives types of local picture of grid diagram near the bands} 
\label{table:localpicture}
\end{table}
\par

Let $A,B,C,D,E$ be the number of bands of type $A,B,C,D,E$ respectively. Then $A+B+C+D+E=b_{1}(F)$ holds.
Since vertical segments of the grid diagram $G_{F}$ appear in pairs with opposite orientation, the other crossings does not affect the writhe. Therefore we have 
\begin{align}
\label{eqn:tb}&tb(\mathcal{L}_{G})= -\frac{1}{2}B -C-\frac{3}{2}D-2E, \\
\label{eqn:rot}&rot(\mathcal{L}_{G})=A+\frac{1}{2}B -\frac{1}{2}D - E +1, \\
\label{eqn:sl}&sl (\mathcal{T}_{G})=tb(\mathcal{L}_{G})-rot(\mathcal{L}_{G})=-(A+B+C+D+E)-1=-b_{1}(F)-1. 
\end{align}
Here the last $+1$ in the rotation number comes from a final modification (Figure~\ref{fig:fpbkgrid}(ii)) to get a grid diagram. 
We remark that we can also prove Theorem~\ref{thm:self-linking} by the equation~$(\ref{eqn:sl})$. 
\par
Let $H$ be the number of horizontal lines in the grid diagram $G_{D}$ which are oriented from right to left. Then we get a closed $H$-braid representing $\mathcal{T}_{G}$ by flipping the horizontal lines oriented from right to left (see Figure~\ref{fig:gridtobraid}). By definition of type $A,B,C,D,E$, we have 
\begin{equation}
H=\frac{1}{2}(4A+3B+2C+D)= b_{1}(F)+rot(\mathcal{L}_{G})-1. 
\end{equation}
\begin{figure}[htbp]
\begin{center}
\includegraphics[width=75mm]{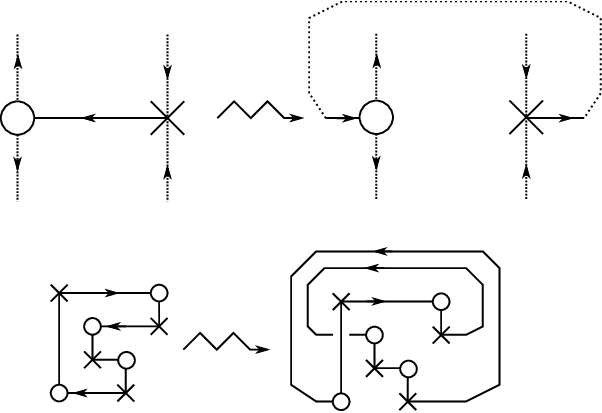}
   \caption{From grid diagram to closed braid diagram}
 \label{fig:gridtobraid}
\end{center}
\end{figure}
Thus we have $tb(\mathcal{L}_{G})=H-2b_{1}(F)$. 
Let $\overline{tb}(L)$ be the maximum Thurston-Bennequin number.
Since $H \geq b(L)$ where $b(L)$ denotes the braid index of $L$, we get the following estimate: 
\begin{thm}[Theorem~\ref{thm:arc-00}]\label{thm:arc1}
For any non-trivial oriented link $L$, we obtain  
\[ 
2fpbk(L) \geq \max\{-\overline{tb}(L),-\overline{tb}(\overline{L})\}+b(L). 
\]
\end{thm}
\par 
It turns out Theorem~\ref{thm:arc1} is quite useful, although it rarely attains the equality compared with the lower bound in Theorem~\ref{thm:self-linking}. As we will see in Section~\ref{sec:tabulation}, Theorem~\ref{thm:arc1} determines the flat plumbing basket number for many knots. 
A key trick is that for a knot $K$, $fpbk(K) \in 2\mathbb{Z}$ hence a weaker inequality like $2fpbk(K)\geq 9$ actually implies a stronger one $fpbk(K)\geq 6$. 
It is interesting to find a condition for the inequality in Theorem~\ref{thm:arc1} to be an equality, or, to find a refinement of Theorem~\ref{thm:arc1} that yields the equality more often.
\section{Tabulation}\label{sec:tabulation}
By using Theorem~\ref{thm:self-linking} and Theorem \ref{thm:arc-00} we improve the table of $fpbk(K)$ for prime knots up to $9$ crossings given in \cite[Table~1]{Hirose-Nakashima} as Tables~\ref{table-1}--\ref{table-last}.
\par
In the table, 
\begin{itemize}
\item The asterisks $^{\ast}$ are improved points.
\item The double asterisks $^{\ast\ast}$ are given by \cite{CCK} and Mikami Hirasawa. He taught the second author that $8_1$ and $9_{44}$ have the flat plumbing basket number $6$.
\item $g(K)$ is the genus. 
\item The check mark $\checkmark$ in the column ``Not monic?'' means that its Alexander polynomial is not monic. Note that for knots with less than or equal to nine crossings $\deg \Delta_{K}(t)=2g(K)$ holds hence it implies that the Hirose-Nakashima bound (\ref{eqn:HN}) gives a better bound $fpbk(K)\geq 2g(K)+4$ than the trivial genus bound (\ref{eqn:genusbound}). 
\item $\alpha(K)$ is the arc index.  
\item $b(K)$ is the braid index. 
\item $TB(K)$ denotes $\max\{-\overline{tb}(K), -\overline{tb}(\overline{K})\}$. 
\item $SL(K)$ denotes $\max\{-\overline{sl}(K), -\overline{sl}(\overline{K})\}$. 
\item The daggers $^{\dagger}$ in the column $g(K)$,``Not monic?'', $\alpha(K)$, $b(K)$ and $SL(K)$ represent that the trivial genus bound (\ref{eqn:genusbound}),  Hirose-Nakashima bound (\ref{eqn:HN}),  the arc index bound (Theorem~\ref{thm:arc0}), the TB$+$braid index bound in Theorem~\ref{thm:arc-00}, and the self-linking number bound in Theorem~\ref{thm:self-linking} determine the $fpbk$, respectively. 
\end{itemize}
\par 
We refer to \cite[Proposition~1.6]{Ng2} for $\overline{sl}(K)$ and \cite{knot_info} for the genus $g(K)$, arc index $\alpha(K)$, $b(K)$, $\overline{tb}(K)$ and $\Delta_{K}(t)$. 
\par 
Here we list some of simple observations.
\par 
\begin{itemize}
\item The lower bound $2fpbk(K)\geq TB(K)+b(K)$ in Theorem~\ref{thm:arc-00} quite often determines the $fpbk(K)$. 
\item The knots $8_{13}$ and $9_{49}$ show that the lower bound $fpbk(K)\geq SL(K)-1$ in Theorem~\ref{thm:self-linking} and the lower bound $2fpbk(K)\geq TB(K)+b(K)$ in Theorem~\ref{thm:arc-00} are independent. 
\item In general, the arc index bound Theorem~\ref{thm:arc0} rarely detects the flat plumbing basket number. 
\item As the knot $6_3$ (resp. $8_4$) demonstrates, sometimes the trivial genus bound (resp. Hirose-Nakashima bound) determines $fpbk(K)$ whereas other lower bounds Theorem~\ref{thm:arc0}, Theorem~\ref{thm:self-linking} and Theorem~\ref{thm:arc-00} cannot.
\end{itemize}
%
%
\begin{table}[hp]
\begin{tabular}{|c||c|c|c|c|c|c|c|}
\hline
$K$		&$g(K)$				& Not monic?						&$\alpha(K)$		& $b(K)$				& $TB(K)$	& $SL(K)$ 			&$fpbk(K)$ \\ \hline\hline
$3_1$		&	1$^{\dagger}$	&										&	5$^{\dagger}$ 	& 2$^{\dagger}$ 	& 6 			&5$^{\dagger}$		&	4			\\ \hline
$4_1$		&	1$^{\dagger}$	&										&	6$^{\dagger}$	& 3$^{\dagger}$ 	& 3 			&	3					&	4			\\ \hline
$5_1$		&	2$^{\dagger}$	&										&	7 					& 2$^{\dagger}$ 	& 10			&	7$^{\dagger}$	&	6			\\ \hline
$5_2$		&	1					&	\checkmark $^{\dagger}$	&	7 					& 3$^{\dagger}$ 	&	8			&7$^{\dagger}$		&	6			\\ \hline
$6_1$		&	1					&  \checkmark $^{\dagger}$	&	8					& 4$^{\dagger}$ 	&	5			&5						&	6			\\ \hline
$6_2$		&	2$^{\dagger}$	&										&	8 					& 3$^{\dagger}$	&	7			&5						&	6			\\ \hline
$6_3$		&	2$^{\dagger}$	&										&	8 					& 3					&	4			&3						&	6			\\ \hline
$7_1$		&	3$^{\dagger}$	&										&	9 					& 2$^{\dagger}$ 	&	14			&9$^{\dagger}$		&  8			\\ \hline
$7_2$		&	1					&\checkmark  						&	9 					& 4$^{\dagger}$ 	&	10			&9$^{\dagger}$		&	8$^{\ast}$			\\ \hline
$7_3$		&	2					&\checkmark $^{\dagger}$		&	9 					& 3$^{\dagger}$	&	12			&9$^{\dagger}$		&	8			\\ \hline
$7_4$		&	1					&\checkmark 						&	9 					& 4$^{\dagger}$ 	&	10			&9$^{\dagger}$		&	8$^{\ast}$			\\ \hline
$7_5$		&	2					&\checkmark $^{\dagger}$		&	9 					& 3$^{\dagger}$ 	&  12			&9$^{\dagger}$		&	8			\\ \hline
$7_6$		&	2$^{\dagger}$	&										&	9$^{\dagger}$	& 4$^{\dagger}$ 	&	8			&7$^{\dagger}$		&	6			\\ \hline
$7_7$		&	2$^{\dagger}$	&										&	9$^{\dagger}$	& 4$^{\dagger}$	&  5			&5						&	6			\\ \hline
$8_1$		&	1					&\checkmark $^{\dagger}$		&	10$^{\dagger}$	& 5$^{\dagger}$ 	&	7			&7$^{\dagger}$		&	6$^{\ast\ast}$			\\ \hline
$8_2$		&	3$^{\dagger}$	&										&	10					& 3$^{\dagger}$	&	11			&7						&	8			\\ \hline
$8_3$		&	1					&\checkmark $^{\dagger}$		&	10$^{\dagger}$	& 5$^{\dagger}$	&	5			&5						&	6			\\ \hline
$8_4$		&	2					&\checkmark $^{\dagger}$		&	10					& 4					&	7			&5						&	8			\\ \hline
$8_5$		&	3$^{\dagger}$ 	&										&  10					& 3$^{\dagger}$	&	11			&7						&	8			\\ \hline
$8_6$		&	2 					&\checkmark $^{\dagger}$		&	10					& 4$^{\dagger}$ 	&	9			&7						&	8			\\ \hline
$8_7$		&	3$^{\dagger}$	&										&	10					& 3					&  8			&5						&	8			\\ \hline
$8_8$		&	2 					&	\checkmark $^{\dagger}$	&	10					& 4 					&	6			&5						&	8			\\ \hline
$8_9$		&	3$^{\dagger}$	& 										&	10					& 3 					&	5			&3						&	8			\\ \hline
$8_{10}$	&	3$^{\dagger}$	&										&	10					& 3 					&	8			&5						&	8			\\ \hline
$8_{11}$	&	2					&\checkmark $^{\dagger}$		&	10					& 4$^{\dagger}$	&	9			&7						&	8			\\ \hline
$8_{12}$	&	2$^{\dagger}$	&										&	10$^{\dagger}$	& 5$^{\dagger}$ 	&	5			&5						&	6			\\ \hline
$8_{13}$	&	2					&\checkmark $^{\dagger}$		&	10					& 4					&	6			&5						&	8			\\ \hline
$8_{14}$	&	2					&\checkmark $^{\dagger}$		&	10					& 4$^{\dagger}$	&	9			&7						&	8			\\ \hline
$8_{15}$	&	2					&\checkmark 						&	10					& 4$^{\dagger}$	&	13			&11$^{\dagger}$	&	10$^{\ast}$			\\ \hline
$8_{16}$	&	3$^{\dagger}$	&										&	10					& 3					&	8			&5						&	8			\\ \hline
$8_{17}$	&	3$^{\dagger}$	&										&	10					& 3					&	5			&3						&	8			\\ \hline
$8_{18}$	&	3$^{\dagger}$	&										&	10					& 3					&	5			&3						&	8			\\ \hline
$8_{19}$	&	3					&										&	7 					& 3 					&	12			&11$^{\dagger}$	&	10$^{\ast}$			\\ \hline
$8_{20}$	&	2$^{\dagger}$	&										&	8 					& 3$^{\dagger}$	&	6			&5						&	6			\\ \hline
$8_{21}$	&	2$^{\dagger}$	&										&	8 					& 3$^{\dagger}$	&	9			&7$^{\dagger}$		&	6			\\ \hline
$9_1$		&	4$^{\dagger}$	&										&	11					& 2$^{\dagger}$	&	18			&11$^{\dagger}$	&	10			\\ \hline
$9_2$		&	1					&\checkmark 						&		11				& 5$^{\dagger}$ 	&	12			&11$^{\dagger}$	&	10$^{\ast}$			\\ \hline
$9_3$		&	3					&\checkmark $^{\dagger}$		&	11					& 3$^{\dagger}$ 	&	16			&11$^{\dagger}$	&	10			\\ \hline
$9_4$		&	2					&\checkmark 						&	11					& 4$^{\dagger}$ 	&	14			&11$^{\dagger}$	&	10$^{\ast}$			\\ \hline
$9_5$		&	1					&\checkmark 						&	11					& 5$^{\dagger}$ 	&	12			&11$^{\dagger}$	&	10$^{\ast}$			\\ \hline

\end{tabular}
\caption{
Table of flat plumbing basket numbers $fpbk(K)$ for prime knots $K$ with up to $9$ crossings.  
For the notations, see Section~\ref{sec:tabulation}. } 
\label{table-1}
\end{table}
\begin{table}[hp]
\begin{tabular}{|c||c|c|c|c|c|c|c|}
\hline
$K$			&$g(K)$					& Not monic?						&$\alpha(K)$			& $b(K)$				& $TB(K)$	& $SL(K)$ 				&$fpbk(K)$ \\ \hline\hline

$9_6$		&	3					&\checkmark $^{\dagger}$		&	11					& 3$^{\dagger}$ 	&	16			&11$^{\dagger}$	&	10			\\ \hline
$9_7$		&	2					&\checkmark 						&	11					& 4$^{\dagger}$ 	&	14			&11$^{\dagger}$	&	10$^{\ast}$			\\ \hline
$9_8$			&	2						&	\checkmark $^{\dagger}$	& 11 						& 5$^{\dagger}$	&8				&	7						&	8			\\ \hline
$9_9$			&	3						&	\checkmark $^{\dagger}$	& 11						& 3$^{\dagger}$	&16			&	11$^{\dagger}$		&	10			\\ \hline
$9_{10}$		&	2						&  \checkmark 					& 11 						& 4$^{\dagger}$	& 14			&	11$^{\dagger}$		&	10$^{\ast}$		\\ \hline
$9_{11}$		&	3$^{\dagger}$		&										&	11						& 4$^{\dagger}$	& 12			& 9$^{\dagger}$		&	8			\\ \hline
$9_{12}$		&	2						& \checkmark $^{\dagger}$		& 11						& 5$^{\dagger}$	& 10			&	9$^{\dagger}$		&	8			\\ \hline
$9_{13}$		&	2						&	\checkmark	 					& 11 						& 4$^{\dagger}$	& 14			&11$^{\dagger}$		&	10$^{\ast}$		\\ \hline
$9_{14}$		&	2						&	\checkmark $^{\dagger}$	& 11 						& 5					&	7			&	7						&8			\\ \hline
$9_{15}$		&	2  					& \checkmark $^{\dagger}$		& 11 						& 5$^{\dagger}$	& 10			&	9$^{\dagger}$		&	8			\\ \hline
$9_{16}$		&	3						&	\checkmark $^{\dagger}$	& 11 						& 3$^{\dagger}$	& 16			&	11$^{\dagger}$		&	10			\\ \hline
$9_{17}$		&	3$^{\dagger}$		&										& 11 						& 4					&	8			& 5						&	8			\\ \hline
$9_{18}$		&	2						&	\checkmark 					& 11 						& 4$^{\dagger}$	&	14			&	11$^{\dagger}$		&	10$^{\ast}$		\\ \hline
$9_{19}$		&	2						&	\checkmark $^{\dagger}$	&	11						& 5 					&6				&	5						&	8			\\ \hline
$9_{20}$		&	3$^{\dagger}$		&										&	11						& 4$^{\dagger}$	&12			&	9$^{\dagger}$		&	8			\\ \hline
$9_{21}$		&	2						&	\checkmark $^{\dagger}$	&  11 					&5$^{\dagger}$	 	&10 			&9$^{\dagger}$			&	8			\\ \hline
$9_{22}$		&	3$^{\dagger}$		&										&	11 					& 4 					&8   			& 5						&	8			\\ \hline
$9_{23}$		&	2						&	\checkmark 					&  11 					& 4$^{\dagger}$	& 14			&	11$^{\dagger}$		&	10$^{\ast}$			\\ \hline
$9_{24}$		&	3$^{\dagger}$		&										&	11						& 4 					& 6  			& 5						&	8			\\ \hline
$9_{25}$		&	2						&	\checkmark 					&  11						& 5	 				& 10 			&	9						&8--10				\\ \hline
$9_{26}$		&	3$^{\dagger}$		&										&	11						& 4$^{\dagger}$	& 9 			& 7						&	8			\\ \hline
$9_{27}$		&	3$^{\dagger}$		&										&	11						& 4 					& 6 			& 5						&	8			\\ \hline
$9_{28}$		&	3$^{\dagger}$		&										&	11						& 4$^{\dagger}$	& 9 			& 7						&	8			\\ \hline
$9_{29}$		&	3$^{\dagger}$		&										&	11						& 4 					& 8 			& 5						&	8			\\ \hline
$9_{30}$		&	3$^{\dagger}$		&										&	11						& 4 					& 6 			&5							&	8			\\ \hline
$9_{31}$		&	3$^{\dagger}$		&										&	11						& 4$^{\dagger}$	& 9 			&7							&	8			\\ \hline
$9_{32}$		&	3$^{\dagger}$		&										&	11						& 4$^{\dagger}$	& 9 			&7							&	8			\\ \hline
$9_{33}$		&	3$^{\dagger}$		&										&	11						& 4 					&6				&5							&	8			\\ \hline
$9_{34}$		&	3						&		         						& 11						& 4 					& 6 			&	5						&8--12				\\ \hline
$9_{35}$		&	1						&	\checkmark 					& 	11						& 5$^{\dagger}$	&	12			&11$^{\dagger}$		&	10$^{\ast}$			\\ \hline
$9_{36}$		&	3$^{\dagger}$		&										&	11						& 4$^{\dagger}$	&	12 		&9$^{\dagger}$			&	8			\\ \hline
$9_{37}$		&	2						& \checkmark $^{\dagger}$		&	11						& 5 					&	6			&	5						&	8			\\ \hline
$9_{38}$		&	2						& \checkmark 						&	11						& 4$^{\dagger}$	&	14			&	11$^{\dagger}$		&	10$^{\ast}$			\\ \hline
$9_{39}$		&	2						& \checkmark 						&	11						& 5 					& 10 			&9							&8--10				\\ \hline
$9_{40}$		&	3						&										&	11						& 4 					& 9			&	7						&8--12				\\ \hline
$9_{41}$		&	2						& \checkmark 						&	11						& 5 					& 7			&7							&8--10				\\ \hline
$9_{42}$		&	2$^{\dagger}$		&										&  8 						& 4$^{\dagger}$	& 5			&	5						&	6			\\ \hline
$9_{43}$		&	3						&										&	9 						& 4 					& 10			& 9						&8--10				\\ \hline
$9_{44}$		&	2$^{\dagger}$		&										&	9$^{\dagger}$		& 4$^{\dagger}$	& 6			&	5						&6$^{\ast\ast}$		\\ \hline
$9_{45}$		&	2						&										&	9 						& 4$^{\dagger}$	&   10		&	9$^{\dagger}$		&8$^{\ast}$			\\ \hline
$9_{46}$		&	1						&	\checkmark $^{\dagger}$	&	8						& 4$^{\dagger}$	& 7			&	7$^{\dagger}$		&	6			\\ \hline
$9_{47}$		&	3$^{\dagger}$		&										&	9	 					& 4 					&	7			&	7						&	8			\\ \hline
$9_{48}$		&	2$^{\dagger}$		&										&	9$^{\dagger}$ 		& 4$^{\dagger}$	&	8			&	7$^{\dagger}$		&	6			\\ \hline
$9_{49}$		&	2						&	\checkmark 					&	9 						& 4 					& 12			&	11$^{\dagger}$		&	10$^{\ast}$			\\ \hline
\end{tabular}
\caption{
Table of flat plumbing basket numbers $fpbk(K)$ for prime knots $K$ with up to $9$ crossings (continuation of Table~\ref{table-1}). } 
\label{table-last}
\end{table}

\noindent{\bf Acknowledgements: }
The second author would like to thank the members of Knotting Nagoya at Nagoya Institute of Technology in June 24--25, 2017. 
In particular, the second author wishes to express his gratitude to Mikami Hirasawa and Susumu Hirose for many helpful comments, data of flat plumbing baskets and their encouragements.  
The first author was supported by JSPS KAKENHI Grant numbers JP15K17540 and JP16H02145.
The second author was supported by JSPS KAKENHI Grant numbers JP16H07230 and JP18K13416. 


\bibliographystyle{amsplain}
\bibliography{mrabbrev,tagami}
\end{document}